\documentclass{article}
\usepackage{graphicx} 
\usepackage{amssymb,amsmath,enumerate,xcolor,amsthm}

\newtheorem{theorem}{Theorem}[section]
\newtheorem{lemma}[theorem]{Lemma}
\newtheorem{corollary}[theorem]{Corollary}

\newtheorem{conjecture}[theorem]{Conjecture}

\DeclareMathOperator{\Pee}{\mathbf{Pr}}
\DeclareMathOperator{\E}{\mathbf{E}}

\title{Asymptotically Optimal Proper Conflict-Free Colouring}
\author{Chun-Hung Liu\footnote{Department of Mathematics, Texas A\&M University, USA. chliu@tamu.edu. Partially supported by NSF under award DMS-1954054 and CAREER award DMS-2144042.},
\quad Bruce Reed \footnote{Mathematical Institute, Academia Sinica, Taiwan. bruce.al.reed@gmail.com.  Supported by  NSTC Grant 112-2115-M-001 -013 -MY3}}

\begin{document}

\maketitle

\begin{abstract}
    A proper conflict-free colouring of a graph is a colouring of the vertices such that any two adjacent vertices receive different colours, and for every non-isolated vertex $v$, some colour appears exactly once on the neighbourhood of $v$.
    Caro, Petru\v{s}evski and \v{S}krekovski conjectured that every connected graph with maximum degree $\Delta \geq 3$ has a proper conflict-free colouring with at most $\Delta+1$ colours.
    This conjecture holds for $\Delta=3$ and remains open for $\Delta \geq 4$.
    In this paper we prove that this conjecture holds asymptotically; namely, every graph with maximum degree $\Delta$ has a proper conflict-free colouring with $(1+o(1))\Delta$ colours.
\end{abstract}

\section{Introduction}

The study of conflict-free colouring of hypergraphs, in which every edge contains a vertex whose colour appears on no other vertex of the edge, was motivated by a frequency assignment problem in cellular networks. 
By choosing hypergraphs appropriately, one can obtain applications to several other colouring notions, such as acyclic colouring, star colouring and frugal colouring.
See \cite{cl} for details.

In this paper our focus is  proper conflict-free colouring of graphs, which were recently introduced by Fabrici, Lu\v{z}ar, Rindo\v{s}ova and Sot\'{a}k \cite{flrs}.  This parameter has attracted wide attention \cite{alo,cps,cckp,cckp2,cl,h,kp,l}.

A {\it proper colouring} of a graph is a colouring of its vertices such that no adjacent vertices receive the same colour.
A {\it proper conflict-free colouring} of a graph $G$ is a proper colouring of $G$ such that for every non-isolated vertex $v$ of $G$, some neighbour of $v$ receives a colour that is not received by any other neighbour of $v$. So, a proper conflict-free colouring of  a graph lies between a proper colouring of   a graph and proper colouring of its square, which requires every  two neighbours of a vertex to  receive different colours.

It is well known  that every graph with maximum degree $\Delta$ can be  properly  coloured with at most $\Delta+1$ colours using a simple greedy procedure. 
On the other hand, there are some graph with maximum degree $\Delta$ for which a proper colouring of the square requires $\Omega(\Delta^2)$ colours. 
It is natural to ask whether the number of colours needed in a proper conflict-free colouring of a graph is closer to $\Delta+1$ or $\Omega(\Delta^2)$. 
Caro, Petru\v{s}evski and \v{S}krekovski \cite{cps} proposed the following conjecture. 

\begin{conjecture}[\cite{cps}] \label{conj_deg}
Every connected graph with maximum degree $\Delta \geq 3$ has a proper conflict-free colouring with at most $\Delta+1$ colours.
\end{conjecture}

The condition $\Delta \geq 3$ in Conjecture \ref{conj_deg} is required since the 5-cycle requires five different colours.
Conjecture \ref{conj_deg} remains open for $\Delta \geq 4$, while the case $\Delta=3$ (even for the list-colouring setting) follows from a result of Liu and Yu \cite{ly} about linear colouring.
Cranston and Liu \cite{cl} showed that every graph with maximum degree $\Delta$ has a proper conflict-free colouring with at most $2\Delta+1$ colours, based on an observation of Pach and Tardos \cite{pt} about conflict-free colouring of hypergraphs.
They \cite{cl} also proved that $(1.6550826+o(1))\Delta$ colours are sufficient even for the list-colouring setting.
Moreover, they \cite{cl} proved that every graph with maximum degree $\Delta$ is fractionally properly conflict-free $(1+o(1))\Delta$-colourable.

Conjecture \ref{conj_deg} is known to hold asymptotically or exactly when extra assumptions are made.
Kamyczura and Przyby\l{}o \cite{kp} proved a result when the graph has large minimum degree: there exists $\Delta_0$ such that every graph of maximum degree $\Delta \ge \Delta_0$ and minimum degree $\delta>1500 \log\Delta$ has a proper conflict-free colouring with at most $\Delta+\frac{600\Delta\log\Delta}{\delta}$ colours.
(The base of $\log$ is $e$ in this paper.)
In addition, Cho, Choi, Kwon and Park \cite{cckp} proved that Conjecture \ref{conj_deg} holds for chordal graphs.

In this paper we prove Conjecture \ref{conj_deg} asymptotically in full generality:

\begin{theorem}
\label{maintheorem}

There exists a positive integer $\Delta_0$ such that if $\Delta \geq \Delta_0$, then every graph of maximum degree at most $\Delta$ has a proper conflict-free colouring with at most $\Delta+100106\Delta^{\frac{2}{3}}\log\Delta$ colours.  

\end{theorem}

The coefficient 100106 can be significantly reduced and we made no attempt to minimize it.

We prove Theorem \ref{maintheorem}  using probabilistic methods.
We review some concentration inequalities in Section \ref{sec:prelim}.
We prove Theorem 1 in Section \ref{sec:proofMain} assuming two main lemmas, and prove these two lemmas in Sections \ref{sec:proofMtheorem} and \ref{sec:proofPartitiontheorem}.

\section{Preliminary} \label{sec:prelim}

In this section, we review some probabilistic tools  that will be used in the rest of the paper.
We denote the expectation and medium of a random variable $X$ by $\E[X]$ and ${\rm Med}(X)$, respectively.

\begin{theorem}[Chernoff Bound (see Page 43 in \cite{mr_book})] \label{chernoff}
Let $n$ be a positive integer.
Let $X_1,X_2,...,X_n$ be random variables such that for each $i \in [n]$, $X_i=1$ with probability $p$, and $X_i=0$ with probability $1-p$.
Let $X = \sum_{i=1}^nX_i$.
Then $\Pee(|X-np| > t) < 2e^{-\frac{t^2}{3np}}$.
\end{theorem}

\begin{theorem}[Lov\'{a}sz Local Lemma \cite{el}] \label{lll}
Let $p<1$ be a nonnegative real number.
Let $d$ be a nonnegative integer.
Let ${\mathcal E}$ be a set of events such that for every $A \in {\mathcal E}$, $\Pee(A) \leq p<1$ and $A$ is independent of a set of all but at most $d$ of the other events.
If $4pd<1$, then with positive probability, none of the events in ${\mathcal E}$ occurs.
\end{theorem}

\begin{theorem}[Azuma's Inequality \cite{a}] \label{azuma}
Let $X$ be a random variable determined by $n$ trials $T_1,T_2,...,T_n$ such that for each $i$, and any two sequences of outcomes $t_1,t_2,...,t_i$ and $t_1,t_2,...,t_{i-1},t_i'$, 
$$|\E[X|T_1=t_1,T_2=t_2,...,T_i=t_i] - \E[X|T_1=t_1,T_2=t_2,...,T_{i-1}=t_{i-1},T_i=t_i']| \leq c_i.$$
Then $\Pee(|X-E[X]|>t) \leq 2e^{-\frac{t^2}{2\sum_{i=1}^nc_i^2}}$.
\end{theorem}

The following is an extension of a result of Talagrand due to McDiarmid,
it is the main result of \cite{mc}:

\begin{theorem}[\cite{mc}]
\label{concentrationtheorem}
Suppose that $X$ is a nonnegative valued random variable on the elements of a product space each component of 
which is a uniformly  random permutation on a finite set and that there are $c,r$ such that swapping the position of any two elements 
in a permutation can change the value of $X$ by at most $c$, and that to certify that $X$ is at least $s$ 
we need only specify the elements in  $rs$ positions where these positions can be in any of the permutations. 
Then
\begin{enumerate}
    \item  for all positive $t$, $\Pee(X>{\rm Med}(X)+t) \le 2e^{\frac{-t^2}{4rc^2({\rm Med}(x)+t)}}$, and  
    \item  for all positive $t$, $\Pee(X<{\rm Med}(X)-t) \le 2e^{\frac{-t^2}{4rc^2{\rm Med}(X)}}$. 
\end{enumerate}
\end{theorem}

We note that we can and do think of a choice of an element from a finite set as a 
permutation of the set where we choose the first element of the permutation. 
We note further that McDiarmid stated the lemma with the hypothesis  that swapping two elements 
changed the outcome by at most $2c$ and hence had a $16$ where we have a four.

This inequality bounds concentration around the median not the mean, so we need to show that these two are close.
We use the following which is Lemma 4.6 in  \cite{MC}  applied to both $X$ and $-X$.

\begin{lemma}
    If the hypothesis of Theorem \ref{concentrationtheorem} hold for a random variable $X$ then 
    $|\E[X]-{\rm Med}(X)|<6c\sqrt{r{\rm Med}(X)}+16c^2r$.
\end{lemma}

Since $\E[X] \ge \frac{{\rm Med}(X)}{2}$ for any nonnegative random variable $X$, we obtain the following corollary.

\begin{corollary}
\label{permutation}
    If the hypotheses of Theorem \ref{concentrationtheorem} hold, then for all  $t>0$, $$\Pee(|X-\E[X]|>t+ 6c\sqrt{2r\E[X]}+16c^2r) \le 4e^{\frac{-t^2}{8rc^2(\E[X]+t)}}.$$ 
\end{corollary}

We refer to this inequality as McDiarmid's Permutations Inequality. 

In this paper, we denote the set of neighbours of a vertex $v$ in a graph $G$ by $N_G(v)$.
We use $N(v)$ when there is no danger of confusion regarding  the host graph $G$.

\section{Proof of Theorem \ref{maintheorem}} \label{sec:proofMain}
We shall prove Theorem \ref{maintheorem} in this section.

We will consider the high degree and low degree vertices separately, using different sets of colours on each.
We set $L=\{v \in V(G): \deg(v) \le \Delta^{1/3}\}$ and $H=V(G)-L$ and apply the following.

\begin{lemma}
\label{Ltheorem}
    There is a proper conflict-free colouring of $G[L]$ using at most $2\Delta^{1/3} + 1$ colours such that for every vertex $v$ of $G$ with at least one neighbour in $L$ there is some $w \in N(v) \cap L$ whose colour appears on no other element of $N(v)$. 
\end{lemma}

\begin{proof}
Denote the vertices in $L$ by $v_1,v_2,...,v_{|L|}$.
We color those vertices by a greedy algorithm.
For every vertex $u \in V(G)$ with $N(u) \cap L \neq \emptyset$, let $\ell(u) = \min\{i \in [|L|]: v_i \in N(u)\}$. 
For $i=1,2,...,|L|$, let $f(v_i)$ be a colour that is not in $\{f(v_j): 1 \leq j \leq i-1, v_j \in N(v_i)\} \cup \{f(v_{\ell(u)}): uv_i \in E(G), \ell(u) \neq v_i\}$.
Since every vertex in $L$ has degree in $G$ at most $\Delta^{1/3}$, at most $2\Delta^{1/3}$ colours are avoided when defining $f(v_i)$.
So the greedy colouring only uses at most $2\Delta^{1/3}+1$ colours.
Moreover, the colouring is clearly proper, and for every vertex $u \in V(G)$ with $N(u) \cap L \neq \emptyset$, $v_{\ell(u)}$ is the only vertex in $N(u) \cap L$ uses the colour $f(v_{\ell(u)})$.
\end{proof}

At this point we simply need to find a proper colouring of $G[H]$ using at most $\Delta+100105\Delta^{2/3}\log\Delta+1$ colours distinct from the colours used in Lemma \ref{Ltheorem} so that for every vertex with all its neighbours in $H$, some neighbour receives a colour used nowhere else in the neighbourhood. 

We next colour a subset $Z$ of $H$ with $\lceil 10^5\Delta^{2/3}\log\Delta \rceil \leq 10^5\Delta^{2/3}\log\Delta+1$ colours so that for every vertex $v$ of $H$ with all its neighbours in $H$ there exists $w \in N(v) \cap Z$ coloured with a colour not used on $N(v)-\{w\}$. 
We further insist that no vertex of $L$ has a completely coloured neighbourhood at the end of this step unless it also satisfies this property. 

\begin{lemma}
\label{Mtheorem}
There is a proper colouring $f$ of $G[Z]$ for some subset $Z$ of $H$ using at most $\lceil 10^5 \Delta^{2/3} \log\Delta \rceil$ colours such that for every vertex $v \in H$ with $N(v) \cap L=\emptyset$ and every vertex $v \in L$ with $N(v) \subseteq f(Z)$, there is some $w \in N(v) \cap Z$ whose colour appears on no other element of $N(v)$.
\end{lemma}

We postpone the proof of Lemma \ref{Mtheorem} to Section \ref{sec:proofMtheorem}.

We now construct a graph $G'$ from $G$ as follows:
    \begin{itemize}
        \item First, delete the vertices coloured in Lemma \ref{Mtheorem} from $H$ to obtain $H'$, and delete from $L$ the vertices which have some neighbour in $L$ or have some neighbour in $H$ coloured with a colour appearing on none of its other neighbours.
	\item Second, for any two vertices in $H'$ with at least $\Delta^{1/3}/\log\Delta$ common neighbours $w$ in $L$ with $|N(w) \cap H'| \leq 50$, we add an edge between them.
        \item Finally, we obtain $L'$ from $L$ by repeatedly doing one of  the following  until it is no longer possible to do so: 
            \begin{itemize}
                \item if there exists a vertex $v$ in $L$ such that $|N(v) \cap H'|=1$ or some vertex in $N(v) \cap H'$ is adjacent to all the other vertices in $N(v) \cap H'$, then delete $v$, and
                \item if there exists a vertex $v$ in $L$ such that some vertex $u$ of $N(v)$ is adjacent to all but exactly one other neighbour $u'$ of $N(v)$, then delete $v$ and add the edge $uu'$. 
            \end{itemize}
    \end{itemize}

Note that $\Delta(G'[H']) \leq \Delta(G[H]) + \frac{50\Delta}{\Delta^{1/3}/\log \Delta} \leq \Delta + 50\Delta^{2/3}\log \Delta$.
Moreover, for every $v \in L'$, $|N_{G'}(v) \cap H'| \geq 3$, and every vertex in $N_{G'}(v) \cap H'$ non-adjacent in $G'$ to at least two other vertices in $N_{G'}(v) \cap H'$.  
Furthermore, for every $v \in L-L'$ with $N(v) \subseteq H$, either some vertex in $N_G(v) \cap H$ has been coloured by a colour not used by other vertices in $N_G(v) \cap H$, or some vertex in $N_{G'}(v) \cap H'$ is adjacent in $G'$ to all other vertices in $N_{G'}(v) \cap H'$.

So, given our  colouring of $L$, to prove the theorem, it suffices to find a proper colouring of $G'[H']$ with at most $\Delta+105\Delta^{\frac{2}{3}}\log\Delta$ colours distinct from the used colours so that for every  vertex $v$ of $L'$ some colour is assigned to exactly one vertex in $N_{G'}(v) \cap H'$. 

Our next step is to partition $H'$ into $t= \lceil 18\Delta^{1/3} \rceil$ subsets which we will colour using disjoint sets of colours. 

We say a vertex $u$ of $H'$ is {\it nearby} a vertex $v$ of $H'$ if $u$ and $v$ have at least 100 common neighbours $w$ in $L'$ with $|N_{G'}(w)| \leq 50$.
For a partition $\{H_i: i \in [t]\}$ of $H'$, we say that a vertex $w \in L'$ is {\it dangerous} (with respect to $\{H_i: i \in [t]\}$) if  for every $u$ in $N(w)$ there exists $i \in [t]$ such that $u \in H_i$ is non-adjacent in $G'$ to a vertex $u'$ in $N_{G'}(w) \cap H_i-\{u\}$ to which it is not nearby, and there exists $u'' \in N_{G'}(w)-(N_{G'}(u) \cup \{u,u'\})$ such that $u'' \in H_j$ for some $j \in [t]$ with $|H_j \cap N_{G'}(w)| \geq 2$. 

\begin{lemma}
\label{partitiontheorem_simple}
We can partition $H'$ into $t= \lceil 18\Delta^{1/3} \rceil$ sets $H_1,H_2,...,H_t$ such that for each $i \in [t]$,
 \begin{enumerate}
     \item for every vertex $v$ in $L'$ with at least $50$ neighbours in $H'$, there exists $i_v \in [t]$ such that $|N_{G'}(v) \cap H_{i_v}|=1$,
     \item for every vertex $v$ of $H_i$,
     \begin{enumerate}
         \item  $v$ has degree in $H_i$ of at most $\frac{\Delta+50\Delta^{2/3}\log \Delta-|N_{G'}(v) \cap L'|}{t}+\Delta^{1/3}\log \Delta$, 
         \item  the number of vertices in  $H_i$ which are nearby $v$  is at most  $\frac{|N_{G'}(v)\cap L'|}{2t}+\Delta^{1/3}\log \Delta$,
         \item  $$\sum_w|(N_{G'}(w)-N_{G'}(v)) \cap H_i| \leq \Delta^{1/3}\log\Delta$$ where the sum is over all dangerous  $w \in L' \cap N_{G'}(v)$ (with respect to $\{H_j: j \in [t]\}$) with at most $50$ neighbours in $H'$. 
     \end{enumerate}
 \end{enumerate}
\end{lemma}

We shall prove Lemma \ref{partitiontheorem_simple} in Section \ref{sec:proofPartitiontheorem}.

Let $\{H_1,H_2,...,H_t\}$ be the partition of $H'$ obtained by Lemma \ref{partitiontheorem_simple}.
Now we construct $G''$ from $G'$ by, for any $i \in [t]$ and  $v \in H_i$, adding (i)   an edge from 
$v$ to every $u \in H_i$ to which it is nearby and (ii) adding all edges between $v$ and every vertex in $(N_{G'}(w)-N_{G'}(v)) \cap H_i$ for every dangerous $w \in N_{G'}(v) \cap L'$ with at most $50$ neighbours in $H'$. 
Lemma \ref{partitiontheorem_simple} tells us that for each $j \in [t]$, we have 
\begin{align*}
    \Delta(G''[H_j]) \le & \max_{v \in H_j}\lfloor (\frac{\Delta+50\Delta^{2/3}\log \Delta-|N_{G'}(v) \cap L'|}{t}+\Delta^{1/3}\log \Delta) \\
    & \ \ \ \ \ \ \ \ \ \ \ + (\frac{|N_{G'}(v)\cap L'|}{2t}+\Delta^{1/3}\log \Delta) + \Delta^{1/3}\log\Delta \rfloor \\
    \leq & \frac{\Delta+50\Delta^{2/3}\log \Delta}{t}+3\Delta^{1/3}\log \Delta.
\end{align*}

We now properly colour each $G''[H_j]$ with $\Delta(G''[H_j])+1$ colours such that the sets of colours used for different $H_j$'s are disjoint. 
Note that the resulting colouring is a proper colouring of $G''[H']$ (and hence of $G'[H']$) using at most $$t \cdot (\frac{\Delta+50\Delta^{2/3}\log \Delta}{t}+3\Delta^{1/3}\log\Delta +1)\leq \Delta+105\Delta^{2/3}\log\Delta$$ colours.

To finish the proof, it suffices to show that for every $w \in L'$, some colour appears in $N_{G'}(w)$ exactly once.
Suppose to the contrary that there exists $w \in L'$ such that no colour appears in $N_{G'}(w)$ exactly once.

By Lemma \ref{partitiontheorem_simple}, $w$ has at most 50 neighbours in $H'$.
We know that $w$ is not dangerous as otherwise we added edges to make $H_i \cap N_{G'}(w)$ a clique for every $i$ and every colour appearing on $H_i \cap N_{G'}(w)$ appears exactly once. 
We choose  $u \in N_{G'}(w)$ which demonstrates  this  and let $h_u \in [t]$ be the index such that $u \in H_{h_u}$.
We know that $u$ is nonadjacent in $G'$ to some vertex $u'$ of $H_{h_u} \cap N_{G'}(w)$ which received the same 
colour. By our construction of $G''$, $u'$ is not nearby $u$. 
By a property of $G'$, $u$ is non-adjacent in $G'$ to at least two  vertices in $N_{G'}(w) \cap H'-\{u\}$.
So there exists $u'' \in N_{G'}(w) \cap H'-\{u,u'\}$ non-adjacent in $G'$ to $u$.
By our choice of $w$, $u'' \in H_j$ for some $j \in [t]$ with $|H_j \cap N_{G'}(w)| \geq 2$. 
But this implies that $u$ demonstrates that $w$ is dangerous, a contradiction. 

This proves Theorem \ref{maintheorem}.
It remains to prove Lemmas \ref{Mtheorem} and \ref{partitiontheorem_simple}.

\section{Proof of Lemma \ref{Mtheorem}} \label{sec:proofMtheorem}

We prove Lemma \ref{Mtheorem} in this section.

For a vertex $v$ of  $H$, we set
$$N^*(v)= (N(v) \cap H) \cup \{w \in H : \exists u \in N(v) \cap N(w) \text{ with } |N(u)| \le 20\}.$$
We note that $|N^*(v)| \le 20\Delta$ since every vertex in $N^*(v)-(N(v) \cap H)$ shares a common neighbour $u$ of degree at most $20$ with $v$.

We first prove the following lemma.

\begin{lemma} \label{goodY}
There exists a subset $Y$ of $H$ such that 
    \begin{enumerate}
        \item there is no vertex $v$ of $L$ with $|N(v) \cap H|>20$ for which $N(v) \cap H \subseteq Y$, and
        \item for every $v \in H$, $360 \log\Delta \leq |N(v) \cap Y| \leq |N^*(v) \cap Y| \leq 25000\Delta^{2/3}\log\Delta$.
    \end{enumerate}
\end{lemma}

\begin{proof}
We place each vertex of $H$ in the set $Y$ independently with probability $p=\frac{800\log \Delta}{\Delta^{1/3}}$. 

For every $v \in L$ with $|N(v) \cap H|>20$, let $A_v$ be the event that $N(v) \cap H \subseteq Y$.
For every $v \in H$, note that $|N(v) \cap Y| \leq |N^*(v) \cap Y|$, and we let $B_v$ be the event that $360 \log\Delta > |N(v) \cap Y|$ or $|N^*(v) \cap Y| > 25000\Delta^{2/3}\log\Delta$.

We will show that with positive probability none of these events occur. 
For each $v \in L$, the events $A_v$ depends only on the assignment of memberships of $Y$ to vertices in $N(v)$. 
For each $v \in H$, the event $B_v$ depends only on the assignment of memberships of $Y$ to vertices in $N^*(v)$. 
It follows that letting $S_v$ be the events indexed  by vertices  at distant at most four from $v$, the event indexed by $v$ is mutually independent of the set of events not in $S_v$. 
So each of these events is mutually independent of a set of all but at most $\Delta^4$ of the other events.
By the Lov\'{a}sz Local Lemma (Theorem \ref{lll}), it suffices to show that the probability of each of these events is at most $\frac{1}{5\Delta^4}$.

We now calculate their probabilities.

For every $v \in L$ with $|N(v) \cap H|>20$, since $p =\frac{800\log \Delta}{\Delta^{1/3}}$ and $|N(v) \cap H| >20$, we have $\Pee(A_v)= p^{|N(v) \cap H|}<\frac{1}{5\Delta^4}$ when $\Delta$ is sufficiently large.

Now we bound $\Pee(B_v)$.
Since for any subset $Z$ of $H$, $\E[|Y \cap Z|] = p|Z|$, so for every $v \in H$, $\E[|Y \cap N(v)|] = p|N(v)| \geq 800\log\Delta$ and $\E[|Y \cap N^*(v)|]=p|N^*(v)| \leq 16000\Delta^{2/3}\log\Delta$.
Furthermore, $|Y \cap Z|$ is determined by  an independent choice (or permutation) for each vertex,
each such choice changes the size of the set by at most one and we can certify that $|Y \cap Z| \ge s$  by specifying $s$ elements in $Z$ which we chose to put in $Y$.
So by McDiarmid's Permutation Inequality (Corollary \ref{permutation}), 
\begin{align*}
    & \Pee(||Y \cap Z| - \E[|Y \cap Z|]| > \frac{1}{2} \cdot \E[|Y \cap Z|]+6\sqrt{2\E[|Y \cap Z|]}+16) \\
    \leq & 4e^{-\frac{(\frac{1}{2}\E[|Y \cap Z|])^2}{8(\E[|Y \cap Z|]+\frac{1}{2}\E[|Y \cap Z|])}} = 4e^{-\frac{\E[|Y \cap Z|]}{48}}.
\end{align*}
Since for large $n$ we also have $6\sqrt{2n}+16 <\frac{n}{20}$, applying this inequality yields that for large $\E[|Y \cap Z|]$,
$$\Pee(||Y \cap Z| - \E[|Y \cap Z|]| \geq (\frac{1}{2}+\frac{1}{20}) \cdot \E[|Y \cap Z|]) \leq 4e^{-\frac{\E[|Y \cap Z|]}{48}}.$$
Taking $Z=N(v)$ and $Z=N^*(v)$ yields that for large $\Delta$,
\begin{align*}
     & \Pee(B_v) \\
     \leq & \Pee(|N(v) \cap Y| < 360\log\Delta) + \Pee(|N^*(v) \cap Y| >25000\Delta^{2/3}\log\Delta) \\
     = & \Pee(|N(v) \cap Y| < \frac{360}{800} \cdot 800\log\Delta) + \Pee(|N^*(v) \cap Y| >\frac{25000}{16000} \cdot 16000\Delta^{2/3}\log\Delta) \\
     \leq & \Pee(|N(v) \cap Y| < \frac{360}{800} \cdot \E[|N(v) \cap Y|]) + \Pee(|N^*(v) \cap Y| >\frac{25000}{16000} \cdot \E[|N^*(v) \cap Y|]) \\
     = & \Pee(|N(v) \cap Y| < (\frac{1}{2}-\frac{1}{20})\E[|N(v) \cap Y|]) + \Pee(|N^*(v) \cap Y| >(\frac{3}{2}+\frac{1}{16})\E[|N^*(v) \cap Y|]) \\
     \leq & \Pee(||Y \cap N(v)| - \E[|Y \cap N(v)|]| > (\frac{1}{2}+\frac{1}{20}) \cdot \E[|Y \cap N(v)|]) \\
      & + \text{   } \Pee(||Y \cap N^*(v)| - \E[|Y \cap N^*(v)|]| > (\frac{1}{2}+\frac{1}{20}) \cdot \E[|Y \cap N^*(v)|]) \\
     \leq & 4e^{-\frac{\E[|Y \cap N(v)|]}{48}} + 4e^{-\frac{\E[|Y \cap N^*(v)|]}{48}} \\
     \leq & 8e^{-\frac{\E[|Y \cap N(v)|]}{48}} \leq 8e^{-\frac{800\log\Delta}{48}} \leq \frac{1}{5\Delta^4}. 
\end{align*}
\end{proof}

Let $c=\lceil 10^5\Delta^{2/3}\log\Delta \rceil$.
Consider the random process where we assign each vertex $y$ of $Y$ a uniform random colour $c_y$ from $\{1,...,c\}$. 
For each $y \in Y$ with no $z \in N^*(y) \cap Y$ with $c_z=c_y$, define $f(y)=c_y$.
By the definition of $N^*$, $f$ does not colour two neighbours of a vertex $u$ with degree at most $20$ with the same colour.
So if for every $v \in H$ with $N(v) \subseteq H$, there exists $y \in N(v) \cap Y$ such that $c_y \neq c_z$ for every $z \in ((N^*(y) \cup N(v))-\{y\}) \cap Y$, then $f$ is a desired colouring for this lemma by Statement 1 of Lemma \ref{goodY}.

For every $v \in H$ with $N(v) \subseteq H$, let $C_v$ be the event that for every $y \in N(v) \cap Y$, $c_y = c_z$ for some $z \in ((N^*(y) \cup N(v))-\{y\}) \cap Y$. 
Hence to prove this lemma, it suffices to show that with positive probability, none of these events occurs.

Note that each $C_v$ only depends on the assignment of colours to  vertices   at distance at most three  from $v$.
So each of these events is mutually independent of a set of all but at most $\Delta^6$ of the other events.
By the Lov\'{a}sz Local Lemma (Theorem \ref{lll}), it suffices to show that the probability of each of these events is at most $\frac{1}{5\Delta^6}$.

Let $v \in H$ with $N(v) \subseteq H$.
To prove $\Pee(C_v) \leq \frac{1}{5\Delta^6}$, it suffices to show that for any choice $I$ of a colouring of $Y-N(v)$, the conditional probability $\Pee(C_v|I) \leq \frac{1}{5\Delta^6}$.

For every $w \in N(v) \cap Y$, we say that $w$ is {\it bad} if there exists $u \in ((N^*(w) \cup N(v))-\{w\}) \cap Y$ with $c_w=c_u$, and we let $D_w$ be the event that $w$ is bad.
For any choice $I$ of a colouring of $Y-N(v)$ and every $w \in N(v) \cap Y$, we have $$\Pee(D_w|I) \leq \frac{|N^*(w) \cap Y|+(|N(v) \cap Y|-1)}{c} \leq \frac{2 \cdot 25000\Delta^{2/3}\log\Delta-1}{c} < \frac{1}{2}$$ by Statement 2 of Lemma \ref{goodY}. 
Therefore, for any choice $I$ of a colouring of $Y-N(v)$, 
$$\E[|\{w \in N(v) \cap Y: w \text{ is bad}\}| \text{   } | \text{  } I] < \frac{1}{2} \cdot |N(v) \cap Y|.$$

Furthermore, given the choice of a colouring of $Y-N(v)$, $\{w \in N(v) \cap Y: w$ is bad$\}$ depends only on our choice of a colour for each vertex in $N(v) \cap Y$, and changing such a choice can affect the value of this random variable by at most two. 
So applying Azuma's inequality (Theorem \ref{azuma}), we obtain:
\begin{align*}
    \Pee(C_v|I) = & P(|\{w \in N(v) \cap Y: w \text{ is bad}\}|=|N(v) \cap Y| \text{   }| \text{   } I) \\
    \leq & 2e^\frac{-(|N(v) \cap Y|/2)^2}{2|N(v) \cap Y| \cdot 2^2} \\
    \leq & 2e^\frac{-|N(v) \cap Y|}{32} \\
    \leq & 2e^{-\frac{1}{32} \cdot 360\log\Delta} \le 2e^{-10\log \Delta} \leq \frac{1}{5\Delta^6}.
\end{align*}
This proves the lemma.

\section{Proof of Lemma \ref{partitiontheorem_simple}} \label{sec:proofPartitiontheorem}

Recall $t= \lceil 18\Delta^{1/3} \rceil$.
We construct the partition $\{H_i: i \in [t]\}$ of $H'$ by considering the random process where each vertex makes an independent and random choice of one of the $t$ parts.
For every $v \in H'$, let $h_v$ be the index in $[t]$ such that $v \in H_{h_v}$.

For $v \in L'$ with at least $50$ neighbours in $H'$, we let $A_v$ be the event that there is no $i \in [t]$ such that  $|N_{G'}(v) \cap H_i|=1$.
For $v \in H'$, we let $B_v$ be the event that $v$ is adjacent in $G'$ to more than $\frac{\Delta+50\Delta^{2/3}\log \Delta -|N_{G'}(v) \cap L'|}{t}+\Delta^{1/3}\log \Delta$ vertices  in $H_{h_v}$, and we let $C_v$ be the event that there are more than $\frac{|N_{G'}(v) \cap L'|}{2t}+\Delta^{1/3}\log \Delta$ vertices in $H_{h_v}$ which are nearby $v$. 
For $v \in H'$, 
    \begin{itemize}
        \item let $S_v$ be the set of vertices $w \in N_{G'}(v) \cap L'$ with at most $50$ neighbours in $H'$, 
        \item let $\mu_v = \sum_{w \in S_v}|(N_{G'}(w)-N_{G'}(v)) \cap H_i|$, and 
        \item let $D_v$ be the event that $\mu_v>\Delta^{1/3}\log \Delta$.
    \end{itemize}
To prove this lemma, it suffices to show that with positive probability, none of these events happens.
We see that these events depend only on the partitioning on the first and second neighbourhoods of $v$.
So for each $v$, there is a set $E_v$ containing at most $3\Delta^4$ events such that the at most three events indexed by $v$ are mutually independent of  the set of events not contained in $E_v$.
By the Lov\'{a}sz Local Lemma (Theorem \ref{lll}), it suffices to show that the probability of each event is at most $\frac{1}{20\Delta^4}$.

\begin{lemma}
For every $v \in L'$ with $|N_{G'}(v)| \geq 50$, $\Pee(A_v) \leq \frac{1}{20\Delta^4}$.
\end{lemma}

\begin{proof}
We let $S_i$ be the set of partitions of $H'$ where there are $i$ partition elements intersecting $N_{G'}(v)$. 
We note that $\Pee(A_v)$ is at most the probability that we choose a partition in $S_j$ for some $j \le \frac{|N_{G'}(v)|}{2}$. 
So $$\Pee(A_v) \leq \frac{\sum_{i\le |N_{G'}(v)|/2} |S_i|}{\sum_{i \leq |N_{G'}(v)|}|S_i|} \leq \frac{\sum_{i\le |N_{G'}(v)|/2} |S_i|}{|S_{\lceil \frac{3}{4}|N_{G'}(v)|\rceil}|}.$$ 
We bound this probability by computing the ratio $\frac{|S_{i+1}|}{|S_i|}$ for $i \le  \frac{3|N_{G'}(v)|}{4} $.
We note that for such $i$ and any partition in the set $S_i$, the number of partition elements containing at least two vertices in $N_{G'}(v)$ is at least $\frac{|N_{G'}(v)|}{4}$. 
We can create an element of $S_{i+1}$ from an element $P$ of $S_i$ by choosing any of the at least $t-|N_{G'}(v)| \geq 17\Delta^{1/3}$ parts disjoint from $N_{G'}(v)$ and moving a vertex $w$ in $N_{G'}(v)$ contained in a nonsingleton part in $P$ to that part. 
Any element of $S_{i+1}$ can be obtained in this way from at most $|N_{G'}(v)|^2$ choices of a partition in $S_i$ obtained by specifying a vertex $w$ of $N_{G'}(v)$ and a part intersecting $N_{G'}(v)-\{w\}$. 
Hence $|S_{i+1}| \ge \frac{17 \Delta^{1/3}|N_{G'}(v)|}{4} \cdot |S_i| \cdot \frac{1}{|N_{G'}(v)|^2} = \frac{17 \Delta^{1/3}}{4|N_{G'}(v)|} \cdot |S_i|>4|S_i|$. 
Thus, $\sum_{i \leq |N_{G'}(v)|/2}|S_i| \le 2|S_{\lfloor |N_{G'}(v)|/2 \rfloor}|$ and 
    \begin{align*}
        |S_{\lceil \frac{3}{4}|N_{G'}(v)|\rceil}| \geq & (\frac{17 \Delta^{1/3}}{4|N_{G'}(v)|})^{\lceil \frac{3}{4}|N_{G'}(v)|\rceil - \lfloor \frac{1}{2}|N_{G'}(v)| \rfloor} |S_{\lfloor \frac{1}{2}|N_{G'}(v)| \rfloor}| \\
        \geq & (\frac{17 \Delta^{1/3}}{4|N_{G'}(v)|})^{\lceil |N_{G'}(v)|/4 \rceil}|S_{\lfloor \frac{1}{2}|N_{G'}(v)| \rfloor}|.
    \end{align*}
Therefore,
$$\Pee(A_v) \leq \frac{\sum_{i\le |N_{G'}(v)|/2} |S_i|}{|S_{\lceil \frac{3}{4}|N_{G'}(v)|\rceil}|} \leq \frac{2|S_{\lfloor \frac{1}{2}|N_{G'}(v)| \rfloor}|}{|S_{\lceil \frac{3}{4}|N_{G'}(v)|\rceil}|} \le 2 \cdot (\frac{4|N_{G'}(v)|}{17 \Delta^{1/3}})^{\lceil |N_{G'}(v)|/4 \rceil}.$$
If $|N_{G'}(v)| \geq 20 \log \Delta$, then $\Pee(A_v)<\frac{1}{20\Delta^4}$ since $\frac{4|N_{G'}(v)|}{17 \Delta^{1/3}}<\frac{1}{4}$; otherwise, for large $\Delta$, $\Pee(A_v)<2 \cdot (\frac{80 \log \Delta}{17 \Delta^{1/3}})^{\lceil N_{G'}(v)/4 \rceil} \leq 2 \cdot (\frac{80 \log \Delta}{17 \Delta^{1/3}})^{13} \le \frac{1}{20\Delta^4}$ since $|N_{G'}(v)| \geq 50$.
This shows that $\Pee(A_v) \leq \frac{1}{20\Delta^4}$.
\end{proof}

\begin{lemma}
For $v \in H'$, $\Pee(B_v) \leq \frac{1}{20\Delta^4}$ and $\Pee(C_v) \leq \frac{1}{20\Delta^4}$.
\end{lemma}

\begin{proof}
We bound the probability of $B_v$ for $v \in H'$.
Let $X_v$ be the number of vertices in $H_{h_v}$ adjacent to $v$ in $G'$.
Note that $\E[X_v] = \frac{|N_{G'}(v) \cap H'|}{t} \le \Delta^{2/3}$ and $\Delta + 50\Delta^{2/3}\log\Delta - |N_{G'}(v) \cap L'| \geq \Delta(G'[H'])- |N_{G'}(v) \cap L'| \geq |N_{G'}(v) \cap H'|$.
By the Chernoff Bound, when $\Delta$ is sufficiently large,
\begin{align*}
    \Pee(B_v) \leq & \Pee(X_v > \frac{\Delta + 50\Delta^{2/3}\log\Delta- |N_{G'}(v) \cap L'|}{t}+\Delta^{1/3}\log\Delta) \\
    \leq & \Pee(X_v > \E[X_v]+\Delta^{1/3}\log \Delta) \\
    \leq & \Pee(|X_v-\E[X_v]| > \Delta^{1/3}\log \Delta) \\
    \leq & 2e^{-\frac{\Delta^{2/3}\log^2\Delta}{3\E[X_v]}} \leq  2e^{-\frac{\log^2\Delta}{3}}  \leq  \frac{1}{20\Delta^4}.
\end{align*}

We bound the probability of $C_v$ for $v \in H'$ in a similar manner.
Let $R_v$ be the number of vertices in $H_{h_v}$ which are nearby $v$.
Note that there are at most $\frac{50|N_{G'}(v) \cap L'|}{100}=\frac{|N_{G'}(v) \cap L'|}{2}$ vertices nearby to $v$. 
Hence $\E[R_v] \le \frac{|N_{G'}(v) \cap L'|}{2t} \le \Delta^{2/3}$.
By the Chernoff Bound, when $\Delta$ is sufficiently large,
\begin{align*}
    \Pee(C_v) \leq & \Pee(R_v > \frac{|N_{G'}(v) \cap L'|}{2t}+\Delta^{1/3}\log\Delta) \\
    \leq & \Pee(R_v > \E[R_v]+\Delta^{1/3}\log\Delta) \\
    \leq & \Pee(|R_v-\E[R_v]| > \Delta^{1/3}\log \Delta) \\
    \leq & 2e^{-\frac{\Delta^{2/3}\log^2\Delta}{3\E[R_v]}} \leq  2e^{-\frac{\log^2\Delta}{3}}  \leq  \frac{1}{20\Delta^4}.
\end{align*}
\end{proof}

Hence it suffices to show  that for every $v \in H'$, $\Pee(D_v) \leq \frac{1}{20\Delta^4}$.
We fix a vertex $v \in H'$.
Note that it suffices to show that $\Pee(D_v|I) \leq \frac{1}{20\Delta^4}$ for any given information $I = (h_u: u \in (N_{G'}(v)  \cup \{v\}) \cap H')$ since $\Pee(D_v) = \sum_I\Pee(D_v,I) = \sum_I\Pee(D_v|I)P(I)$.

Let $I = (h_u: u \in (N_{G'}(v) \cup \{v\}) \cap H')$.
Note that for every $w \in N_{G'}(v)  \cap L'$ with $|N_{G'}(w)| \leq 50$, if $w$ is dangerous, then there exists $z \in N_{G'}(w) - N_{G'}(v)$ contained in the same part as $v$, and there exists a vertex $z' \in N_{G'}(w)-(N_{G'}(v) \cup \{v,z\})$ contained in a part as a vertex in $N_{G'}(w)-\{z'\}$, so 
    \begin{align*}
         & \Pee(w \text{ is dangerous }|I) \\
        \leq & (|N_{G'}(w)-N_{G'}(v)| \cdot \frac{1}{t}) \cdot (|N_{G'}(w)-N_{G'}(v)|-2) \cdot \frac{|N_{G'}(w)|-1}{t} \\
        \leq & |N_{G'}(w)|^3 \cdot \frac{1}{t^2} \leq \frac{50^3}{t^2};
    \end{align*} 
moreover, $w$ contributes at most $|N_{G'}(w)| \leq 50$ to $\mu_v$.
So $\E[\mu_v|I] \leq |S_v| \cdot 50 \cdot \frac{50^3}{t^2}$. 
Since $|S_v| \leq \Delta$, we have $\E[\mu_v|I] \leq 50^3 \Delta^{1/3}$. 

\begin{lemma}
We next expose the choice of partition elements containing the vertices in $\bigcup_{w \in S_v}N_{G'}(w)$ which are nearby but not adjacent in $G'$ to $v$. 
Then with probability at least $1-\frac{1}{40\Delta^4}$, the conditional expectation of $\mu_v$ given these choice is at most $50^5\Delta^{1/3}$.
\end{lemma}

\begin{proof}
This expectation comes from two subsets of vertices of $S_v$. 
One subset $Z_v$ consists of those $w$ in $S_v$ for which at the end of the exposure process there are still no nonneighbours of $v$ in $N_{G'}(w)$ which are assigned to the same partition element as another neighbour of $w$.  
For these $w$, the probability that $w$ is dangerous is still at most $\frac{|N_{G'}(w)|^3}{t^2} \leq \frac{50^3}{t^2}$ as before. 
Since each such $w$ contributes at most 50 to the conditional expectation of $\mu_v$, we know that the vertices in $Z_v$ contribute at most $|Z_v| \cdot 50 \cdot \frac{50^3}{t^2} \leq \Delta \cdot \frac{50^4}{18^2\Delta^{2/3}} \leq 50^4\Delta^{1/3}$ to the conditional expectation of $\mu_v$ in total.

The second subset $W_v$ consists of those $w$ in $S_v$ for which some (nearby) nonneighbour of $v$ in $N_{G'}(w)$ is assigned to the same partition element as another element of $N_{G'}(w)$. 
For these $w$, if $w$ is dangerous, then there exists a nonneighbour of $v$ in $N_{G'}(w)$ not nearby $v$ assigned to the same partition element as $v$, so the probability that $w$ is dangerous is at most $\frac{|N_{G'}(w)|}{t} \leq \frac{50}{t}$. 
So the conditional expectation that such a $w$ contributes to $\mu_v$ is at most $\frac{50^2}{t}$.
Hence the contribution to $\mu_v$ for all of $W_v$ is at most $\frac{50^2|W_v|}{t}$. 

So, our objective is to show that $50^4\Delta^{1/3} + \frac{50^2|W_v|}{t} \leq 50^5\Delta^{1/3}$ with probability at least $1-\frac{1}{40\Delta^4}$.
It suffices to show $\Pee(|W_v| \le 50^3\Delta^{2/3}) \geq 1-\frac{1}{40\Delta^4}$.

For a vertex $w$ in $S_v$, if $w \in W_v$, then some (nearby) nonneighbour of $v$ in $N_{G'}(w)$ is assigned to the same partition element as another element of $N_{G'}(w)$; so $\Pee(w \in W_v) \leq |N_{G'}(w)| \cdot \frac{|N_{G'}(w)|}{t} \leq \frac{50^2}{t}$.
So $\E[|W_v|] \leq |S_v| \cdot \frac{50^2}{t} \leq \frac{50^2\Delta^{2/3}}{18}$. 

Note that $|W_v|$ is determined by trials $T_u$'s (for $u \in \bigcup_{w \in S_v}N_{G'}(w)-N_{G'}(v)$ that are nearby $v$), where each $T_u$ says which partition element contains $u$.
For each such $u$, let $c_u$ be the maximum by which  changing the outcome of $T_u$ can change  $|W_v|$.
Then $c_u$ is at most the number of common neighbours of degree at most $50$ which $u$ has with $v$. 
By the way in which we constructed $G'$, $u$ has at most $\frac{\Delta^{1/3}}{\log \Delta}$ common neighbours with degree at most $50$ with $v$, so each $c_u$ is at most $\frac{\Delta^{1/3}}{\log \Delta}$.
Furthermore, the sum of the $c_u$'s is at most the sum of the size of these common neighbourhoods, so it is at most $50\Delta$.  
Hence $\sum_uc_u^2 \leq \max_uc_u \cdot \sum_uc_u \leq \frac{\Delta^{1/3}}{\log \Delta} \cdot 50\Delta = \frac{50\Delta^{4/3}}{\log\Delta}$. 
So, via an application of Azuma's inequality:

\begin{align*}
    \Pee(|W_v|>50^3\Delta^{2/3}) \leq & \Pee(W_v > E[W_v]+\frac{50^3\Delta^{2/3}}{2}) \\
    \leq & \Pee(|W_v-E[W_v]| > \frac{50^3\Delta^{2/3}}{2}) \\
    \leq & 2e^{-\frac{50^6\Delta^{4/3}/4}{2 \cdot 50\Delta^{4/3}/\log \Delta}} \leq  2e^{-100\log\Delta}  
    \leq  \frac{1}{40\Delta^4}.
\end{align*}
\end{proof}

So, to complete our proof that $\Pee(D_v) \le \frac{1}{20\Delta^4}$, it suffices to show that $\Pr(\mu_v> \Delta^{1/3}\log \Delta \ |Q) \leq \frac{1}{40\Delta^4}$ given a choice $Q$ of partition elements for the vertices in $\bigcup_{w \in S_v}N_{G'}(w)$ which are nearby but not adjacent in $G'$ to $v$ such that $\E[\mu_v|Q] \leq 50^5 \Delta^{1/3}$. 

We show it using a concentration argument. 
Now $\mu_v$ is determined by our choice of partition elements for the nonneighbours of $v$ which are not nearby to it.  
Furthermore, each such choice can change $\mu_v$ by at most $50 \cdot 100=5000$ and to certify that $\mu_v \geq s$ we need only specify at most $50s$ choices. 
Finally, the median of $\mu_v$ is at most $2\E[\mu_v|Q] \leq 2 \cdot 50^5\Delta^{1/3}$.
Hence applying Theorem \ref{concentrationtheorem}, we obtain for large enough $\Delta$

\begin{align*}
    \Pee(\mu_v>\Delta^{1/3}\log \Delta) 
    \leq &  2e^{\frac{-(\Delta^{1/3}\log \Delta-2 \cdot 50^5\Delta^{1/3})^2}{4 \cdot 50 \cdot 5000^2(\Delta^{1/3}\log \Delta)}} \\
    \leq &  2e^{\frac{-(\frac{\Delta^{1/3}\log \Delta}{2})^2}{4 \cdot 50 \cdot 5000^2(\Delta^{1/3}\log \Delta)}} \leq \frac{1}{40\Delta^4}.
\end{align*}

This proves the lemma.

\bigskip

\bigskip

\noindent{\bf Acknowledgement:} 
This work was conducted when the first author visited the Institute of Mathematics at Academia Sinica at Taiwan. He thanks the Institute of Mathematics at Academia Sinica for its hospitality.


\begin{thebibliography}{10}

\bibitem{alo}
J. Ahn, S. Im and S. Oum, {\it The proper conflict-free $k$-coloring problem and the odd $k$-coloring problem are NP-complete on bipartite graphs}, arXiv:2208.08330.

\bibitem{a}
K. Azuma, {\it Weighted sums of certain dependent random variables}. Tohoku Math. Journal 19 (1967), 357--367. 

\bibitem{cps}
Y. Caro, M. Petru\v{s}evski and R. \v{S}krekovski, {\it Remarks on proper conflict-free colorings of graphs}. Discrete Math. 346 (2023), 113221. 

\bibitem{cckp}
E.-K. Cho, I. Choi, H. Kwon and B. Park, {\it Brooks-type theorems for relaxations of square colorings}, Discrete Math. 348 (2025), 114233.

\bibitem{cckp2}
E.-K. Cho, I. Choi, H. Kwon and B. Park, {\it Proper conflict-free coloring of sparse graphs}, Discrete Appl. Math. 362 (2025), 34--42.

\bibitem{cl}
D. W. Cranston and C.-H. Liu, {\it Proper conflict-free coloring of graphs with large maximum degree}, SIAM J. Discrete Math. 38 (2024), 3004--3027. 

\bibitem{el}
P. Erd\H{o}s and L. Lov\'{a}sz, {\it Problems and results on 3-chromatic hypergraphs and some related questions}. In: ``Infinite and Finite Sets'' (A. Hajnal et al.\ Eds), Colloq. Math. Soc. J. Bolyai 11, North Holland, Amsterdam, 609--627 (1975). 

\bibitem{flrs}
I. Fabrici, B. Lu\v{z}ar, S. Rindo\v{s}ova and R. Sot\'{a}k, {\it Proper conflict-free and unique-maximum colorings of planar graphs with respect to neighborhoods}, Discrete Appl. Math. 324 (2023), 80--92.

\bibitem{h}
R. Hickingbotham, {\it Odd colourings, conflict-free colourings and strong colouring numbers}, Australas. J. Combin. 87 (2023), 160--164.

\bibitem{kp}
M. Kamyczura and J. Przyby\l{}o, {\it On conflict-free proper colourings of graphs without small degree vertices}, Discrete Math. 347 (2024), 113712.

\bibitem{l}
C.-H. Liu, {\it Proper conflict-free list-coloring, odd minors, subdivisions, and layered treewidth}, Discrete Math. 347 (2024), 113668.

\bibitem{ly}
C.-H. Liu and G. Yu, {\it Linear colorings of subcubic graphs}, European J. Combin. 34 (2013), 1040--1050.

\bibitem{MC}
C. McDiarmid, {\it Concentration} pp. 195-248 in Probabilistic methods for Algorithmic Discrete Mathematics 
(Eds: Habib, McDiarmid, Ramirez-Alfonsin,Reed), Springer-Verlag, Berlin(1998). 

\bibitem{mc}
C. McDiarmid, {\it Concentration of Independent Permutations}, Combinatorics, Probability and Computing 11(2002), 163-178. 

\bibitem{mr_book}
M. Molloy and B. Reed, Graph Colouring and the Probabilistic Method, Springer-Verlag Berlin Heidelberg (2002).

\bibitem{pt}
J. Pach and G. Tardos, {\it Conflict-free colourings of graphs and hypergraphs}, Combin. Probab. Comput. 18 (2009), 819--834.


\end{thebibliography}
\end{document}